\documentclass[12pt,a4paper]{amsart}
\usepackage{graphicx,amssymb}
\input xy
\xyoption{all}
\usepackage[all]{xy}
\usepackage{hyperref}

\newtheorem{thm}{Theorem}[section]
\newtheorem{cor}[thm]{Corollary}
\newtheorem{lem}[thm]{Lemma}
\newtheorem{prop}[thm]{Proposition}
\theoremstyle{definition}

\numberwithin{equation}{section}

\newcommand{\PP}{\mathbb P}

\newcommand{\lra}{\longrightarrow}

\newcommand{\ra}{\rightarrow}

 \DeclareMathOperator{\Ker}{Ker}

 \DeclareMathOperator{\Nm}{{Nm}}

\begin{document}

\title[ ]{Prym varieties of \'etale covers of hyperelliptic curves}%
\author{ Herbert Lange and  Angela Ortega}
\address{H. Lange \\ Department Mathematik der Universit\"at Erlangen \\ Germany}
\email{lange@math.fau.de}
              
\address{A. Ortega \\ Institut f\" ur Mathematik, Humboldt Universit\"at zu Berlin \\ Germany}
\email{ortega@math.hu-berlin.de}

\thanks{The second author was supported by Deutsche Forschungsgemeinschaft, SFB 647.}

\begin{abstract}
It is well known that the Prym variety of an \'etale cyclic covering of a hyperelliptic curve is isogenous to the 
product of two Jacobians. Moreover, if the degree of the covering is odd or congruent to 2 mod 4, then the canonical 
isogeny is an isomorphism. We compute the degree of this isogeny in the remaining cases and show 
that only in the case of coverings of degree 4 it is an isomorphism.   
\end{abstract}
\maketitle

\section{Introduction}

Let $H$ denote a hyperelliptic curve of genus $g \geq 2$ and $f: X \ra H$ an  \'etale cyclic covering 
of degree $n \geq 2$. Let $\sigma$ denote the automorphism of $X$ defining $f$. It is well known 
that the hyperelliptic involution of $H$ lifts to an involution $\tau$ on $X$. Then $\sigma$ and $\tau$ 
generate the dihedral group $D_n$ of order $2n$. The Prym variety $P(f)$ of $f$ is defined as the 
connected component containing 0 of the kernel of the norm map $\Nm f: JX \ra JH$ of $f$. 
For any element $\alpha \in D_n$ we denote by $X_\alpha$ the quotient of $X$ by the subgroup generated by $\alpha$. The Jacobians $JX_\tau$ and $JX_{\tau \sigma}$ are abelian subvarieties of
the Prym variety $P(f)$ so the the addition map
$$
a: JX_\tau \times JX_{\tau \sigma} \ra P(f)
$$
is well defined. Mumford showed in \cite{m} that for $n = 2$ the map $a$ is an isomorphism. J. Ries proved the 
same for any odd prime degree $n$ (\cite{ri}). The second author generalized this statement largely to show that
$a$ is an isomorphism for any odd number and, more important, for any even $n \equiv 2 \mod 4$ (\cite{o}). 

It is an obvious question  whether this is true for any degree $n$. In fact, using the action of the 
group $D_n$ on $P(f)$ and a little representation theory, it is not difficult to see that $a$ is an isogeny.
For more precise results on the decomposition of $P(f)$ up to isogeny see \cite{crr}. It is the aim of 
this note to compute the degree of the isogeny $a$. Our main result is\\

{\bf Theorem 4.1}
{\it Let $f:X \ra H$ be as above with $n = 2^rm,  \; r \geq 2$ and $m$ odd. Then $a$ is an isogeny of 
degree 
$$
\deg a = 2^{[(2^r -r -1)m - (r-1)](g-1)}.
$$}

\noindent
So $a$ is an isomorphism for odd $n$, for $n \equiv 2 \mod 4$, and for $n = 4$. 
The proof proceeds by induction on the exponent $r$, the beginning of the induction being Ortega's theorem in \cite{o}.

\section{Preliminaries} \label{prel}

 Let $H$ be a smooth hyperelliptic curve of genus $g$ 
with hyperellptic covering $\pi: H \ra \PP^1$ and $f:X \ra H$ be a cyclic \'etale covering 
of degree $n \geq 2$. So $X$ is of genus $g_X = n(g-1) + 1$ and the Prym variety $P := P(f)$ 
of $f$ is an abelian
variety of dimension 
\begin{equation} \label{eq1.1}
\dim P = (n-1)(g-1).
\end{equation}
The canonical polarization of $JX$ induces a polarization on $P$
of type 
$$
( \underbrace{1,\dots, 1}_{(n-2)(g-1)}, \underbrace{n,\dots,n}_{g-1}).
$$ 
The hyperelliptic involution of $H$ lifts to an involution $\tau$ on $X$ which together with the
 automorphism $\sigma$ defined by the covering $f$ generate the dihedral group
$$
D_n := \langle  \sigma, \tau \;|\; \sigma^n = \tau^2 = (\sigma\tau)^2 = 1 \rangle.
$$
The automorphism $\sigma$  induces an automorphism of the same order $n$ of $P$  compatible 
with the polarization, which we denote by the same letter. Each eigenvalue 
$\zeta_n^i, i=1, \dots, n-1$ (with $\zeta_n$  a fixed primitive $n$-th root of unity)
of the induced map on the tangent space $T_0P$ occurs with multiplicity $g-1$.

In the whole paper we write
$$
n = 2^r m
$$
with $r \geq 0$ and $m$ odd.

In any case the group $D_n$ admits $n$ involutions, namely $\tau \sigma^\nu$ for 
$\nu = 0, \dots n-1$. For odd $n$, these are all the
involutions. For even $n$, there is one more, namely $\sigma^{\frac{n}{2}}$. For odd $n$ all 
involutions are conjugate to $\tau$ and for even $n$ there are 3 conjugacy classes. They are  represented by
$$
\tau, \; \tau \sigma^m\quad \mbox{and} \quad \sigma^{n/2}.
$$
 For any subgroup $S \subset D_n$ and for any element $\alpha \in D_n$ we denote by
$$
X_S:= X/S \quad \mbox{and} \quad X_\alpha := X/ \langle \alpha \rangle 
$$
the corresponding quotients.

Consider the following diagram (for odd $n$ only the left hand side of the diagram, since in this case $m=n$, so both sides are the same).
$$
\xymatrix{
& X \ar[d]_f^{n:1} \ar[dr]^{2:1} \ar[dl]_{2:1} &\\
X_{\tau} \ar[dr]_{n:1}  & H \ar[d]_\pi^{2:1} &  X_{\tau\sigma^m}  \ar[dl]^{n:1}\\
& \PP^1 &
}
$$
Let $W$ denote the set of $2g+2$ branch points of the hyperelliptic 
covering $\pi$. Then denote for arbitrary $n$, 
$$
s_0 := \left| \{ x \in W \;|\; (\pi f)^{-1}(x) \; \mbox{contains a fixed point of}\; \tau  \} \right|
$$
and
$$
s_1 := \left| \{ x \in W \;|\; (\pi f)^{-1}(x) \; \mbox{contains a fixed point of}\; \tau \sigma^m \} \right|.
$$

According to \cite[Proposition 2.4]{o}
the Jacobians $JX_\tau$ and $JX_{\tau \sigma^m}$ are contained in the Prym variety $P$. 
With these notations the following theorem is proved in \cite{o}.

\begin{thm} \label{thm2.1} {\em (a)}
 For odd $n$ the map 
$$
\psi: (JX_\tau)^2 \ra P, \qquad (x,y) \mapsto x + \sigma(y)
$$
is an isomorphism.

{\em (b)} 
For $n = 2m \equiv 2 \mod 4$ the map
$$
\psi: JX_\tau \times JX_{\tau\sigma^m} \ra P, \qquad (x,y) \mapsto x + y
$$
is an isomorphism. Moreover, 
$$
g(X_\tau )= m(g-1) + 1 - \frac{s_0}{2} \quad \mbox{and} \quad 
g(X_{\tau \sigma^m})= m(g-1) + 1 - \frac{s_1}{2}.
$$
In particular $s_0$ and $s_1$ are even.
\end{thm}

It is the aim of this paper to study the map $\psi$ in the remaining cases $n = 2^rm$ with $r \geq 2$. 
So in the sequel we assume $r \geq 2$. We first need some preliminaries.\\

There are 2 non-conjugate Kleinian subgroups of $D_n$, namely
$$
K_\tau = \{1, \sigma^{n/2}, \tau,  \tau \sigma^{n/2}\} \quad \mbox{and} \quad 
K_{\tau \sigma^m} = \{1, \sigma^{n/2}, \tau \sigma^m, \tau  \sigma^{m+n/2} \}.
$$
Moreover, consider the dihedral subgroups of order 8,
$$
T_\tau = \langle \tau, \sigma^{n/4} \rangle \quad \mbox{and} \quad 
T_{\tau \sigma^{m}}= \langle \tau \sigma^m, \sigma^{n/4} \rangle.
$$
Note that for $r \geq 3$ the groups $T_\tau$ and $T_{\tau \sigma^{m}}$ are non-conjugate, whereas
\begin{equation} \label{e1.2}
T_\tau = T_{\tau \sigma^{m}} \qquad \mbox{for} \qquad r=2,
\end{equation} 
since then $\frac{n}{4} = m$ and $\langle \tau, \sigma^m \rangle =   \langle \tau \sigma^m, \sigma^m  \rangle$.
In any case we have the following commutative diagram

\begin{equation} \label{diag1.2}
\xymatrix@R=1cm@C=1cm{
& X \ar[d]^{2:1}_{f_1} \ar[dr]^{a_{\tau \sigma^m}} \ar[dl]_{a_{\tau \sigma^{n/2}}} &\\
X_{\tau \sigma^{n/2}} \ar[d]_{b_{\tau\sigma^{n/2}}} & X_{\sigma^{n/2}} \ar[d]^{2:1}_{f_2} \ar[dl]_{c_{\tau \sigma^{n/2}}} \ar[dr]^{c_{\tau \sigma^m}} & X_{\tau \sigma^m} \ar[d]^{b_{\tau \sigma^m}} \\
X_{K_\tau} \ar[d]_{d_{\tau \sigma^{n/2}}} & X_{\sigma^{n/4}} \ar[dl]_{e_{\tau \sigma^{n/2}}} \ar[dr]^{e_{\tau \sigma^m}} \ar[dd]^{\frac{n}{4}:1}_{f_3}& X_{K_{\tau\sigma^m}} \ar[d]^{d_{\tau \sigma^m}} \\
X_{T_\tau} \ar[ddr]_{\frac{n}{4}:1} && X_{T_{\tau \sigma^m}} \ar[ddl]^{\frac{n}{4}:1}\\
&  X_{\sigma} = H \ar[d]^\pi &\\
& \PP^1 &
}
\end{equation}

In the sequel we use the following notation: if an involution of the group $D_n$ induces an involution on a curve of the diagram, we denote the induced involution by the same letter.
In order to compute the genera of the curves in the diagram, we need the following lemma.

\begin{lem} \label{lem1.2}
Suppose that the dihedral group $D_n = \langle \sigma, \tau \rangle$ of 
order $2n$ with $n \geq 3$ acts on a finite set $S$ of $n$ elements such that the subgroup 
$\langle \sigma \rangle$ acts transitively on $S$. 

{\em (a)} If $n$ is odd, $\tau$ admits exactly one fixed point,

{\em (b)} for even $n$,  either $\tau$ acts fixed-point free or admits exactly $2$ fixed points.

{\em (c)} for even $n$, exactly one of the involutions
$\tau$ and $\tau \sigma$  admits a fixed point.
\end{lem}

\begin{proof}
Let $S = \{x_1, \dots x_n\}$. We may enumerate the $x_i$ in such a way that
$\sigma(x_i) = x_{i+1}$ for $i=1, \dots,n$ where $x_{n+1} = x_1$.
If $n$ is odd, then clearly $\tau$ admits a fixed point. So in any case we may assume that $x_1$ is a 
fixed point of $\tau$. Then we have inductively for $i = 1, \dots \lfloor \frac{n+3}{2} \rfloor$,
\begin{equation} \label{eq4.1}
\tau(x_i) = x_{n+2-i}.
\end{equation}
In fact, the induction step is $\tau(x_i) = \tau \sigma(x_{i-1}) = \sigma^{-1} \tau(x_{i-1}) = 
\sigma^{-1}(x_{n-i+3}) = x_{n-i+2}$. Hence for odd $n$ the involution
$\tau$ admits no further fixed point and for even $n$ $\tau$ admits exactly one additional fixed point, 
namely $x_{\frac{n+2}{2}}$. This gives (a) and (b).

(c): Suppose $n$ is even and $\tau$ admits a fixed point, say $x_1$. Hence we have \eqref{eq4.1}
for all $i$. This implies 
$$
\tau \sigma(x_i) = \tau(x_{i+1}) = x_{n+1-i}.
$$
and $\tau \sigma$ acts fixed point free. Conversely, suppose $\tau$ acts fixed point free. Suppose that
$\tau(x_1) = x_i$ for some $i \geq 2$. Then $\sigma^{1-i} \tau (x_1) = x_1$. So 
$\sigma^{1-i} \tau$ admits a fixed point and thus cannot be equivalent to $\tau$. Hence $\tau \sigma$ is equivalent to $\sigma^{1-i} \tau$
and admits a fixed point.
\end{proof}

\begin{lem} \label{lem1.3} Suppose $n = 2^rm$ with $m$ odd and $r \geq 2$. Then
 
{\rm(i)}
$$
s_0 + s_1 = 2g+2 \quad \mbox{with} \quad s_0, s_1 \geq 2 \; \mbox{even};
$$

{\rm (ii)} for $r=2$, $X_{\sigma^{n/4}} \ra X_{T_\tau}$  and $X_{\sigma^{n/4}} \ra  X_{T_{\tau \sigma^m}}$
are ramified exactly at $2g+2$ points.

{\rm(iii)} $X \ra X_\tau$ and $X_{\sigma^{n/2}} \ra X_{K_\tau}$ as well as 
$X_{\sigma^{n/4}} \ra X_{T_\tau} $, if $r \geq 3$,  
are ramified exactly at $2s_0$
 points.  $X \ra X_{\tau \sigma^m}$ and $X_{\sigma^{n/2}} \ra X_{K_{\tau \sigma^m}}$ 
as well as $X_{\sigma^{n/4}} \ra  X_{T_{\tau \sigma^m}}$, if $r \geq 3$,
are ramified exactly at $2s_1$  points.
\end{lem}

\begin{proof}
The fixed points of $\tau$ and $\tau \sigma^m$ lie over the $2g+2$ Weierstrass points of $H$. Moreover, according to Lemma \ref{lem1.2}. over each Weierstrass point of $H$
exactly one of $\tau$ and $\tau \sigma^m$ admits a fixed point. This gives the first assertion of (i).
The evenness of $s_0$ and $s_1$ follows from the Hurwitz formula. Now $s_0 =0$ means that 
$\tau$ acts fixed-point free. Since also $\sigma$ acts fixed-point free, so does $\tau \sigma^m$ which 
means $s_1 = 0$. But this contradicts the equation $s_0 + s_1 = 2g+2$. Hence $s_0, s_1 \geq 2$.

 If $x$ is a Weierstrass point of $H$ and $\tau$ admits a fixed point over $x$,
then $D_n$ acts on the fibre $f^{-1}(x)$. Similarly, the group $D_{n/2} = \langle \sigma^{n/2}, \tau \rangle$ acts
 on the fibre $(f_3 \circ f_2)^{-1}(x)$ and the group $D_{n/4} = \langle \sigma^{n/4}, \tau \rangle$ acts of the fibre $f_3^{-1}(x)$.
Hence Lemma \ref{lem1.2} implies (ii), since in these cases the order of the fibre is even,  and (iii),
since in this case the order of the fibre is odd.
\end{proof}

By checking the ramification of the maps in diagram \eqref{diag1.2} we immediately get from
Lemma \ref{lem1.3} the following corollaries.

\begin{cor} \label{cor1.3}
All vertical left and right hand maps are ramified. 
\end{cor}

\begin{cor} \label{cor1.4} If $n =2^rm$ with $m$ odd and $r \geq 2$, then
$$
g(X) = n(g-1)+1 \qquad g(X_{\sigma^{n/2}}) = \frac{n}{2}(g-1) +1 \qquad 
g(X_{\sigma^{n/4}}) = \frac{n}{4}(g -1) +1;
$$
$$
g(X_{\tau \sigma^{n/2}}) = \frac{n}{2} (g-1) + 1 - \frac{s_0}{2} \qquad  \qquad  
g(X_{\tau \sigma^m}) = \frac{n}{2}(g-1) +1  - \frac{s_1}{2};
$$
$$
g(X_{K_\tau}) = \frac{n}{4}(g-1) + 1 - \frac{s_0}{2} \qquad \qquad   
g(X_{K_{\tau \sigma^m}}) = \frac{n}{4}(g-1) + 1 - \frac{s_1}{2};
$$
and for $r \geq 3$,
$$
g(X_{T_\tau}) = \frac{n}{8}(g-1) +1 - \frac{s_0}{2}\qquad \qquad 
g(X_{\tau \sigma^m}) = \frac{n}{8}(g-1) +1 - \frac{s_1}{2}.
$$
For $r= 2$, 
$$
g(X_{T_\tau}) = g(X_{\tau \sigma^m}) = \frac{1}{2}(m-1)(g-1).
$$
\end{cor}

\begin{proof}
All assertions follow from the Hurwitz formula. For the first line of assertions we use the fact that $f$ 
is \'etale. For the other formulas we use Lemma \ref{lem1.3}(ii) and (iii). 
\end{proof}

The following lemma is well known. 
In fact, it is an easy consequence of \cite[Proposition 11.4.3]{bl} and \cite[Corollary 12.1.4]{bl}.

\begin{lem} \label{lem1.5}
Let $g: Y \ra Z$ be a covering of smooth projective curves of degree $d \geq 2$.
The addition map 
$$
g^*JZ \times P(Y/Z) \ra JY
$$
is an isogeny of degree 
$$
|g^*JZ \cap P(Y/Z)| = \frac{|JZ[d]|}{|\ker g^*|^2}.
$$

\end{lem}

We need a result on curves with an action of the Klein group. Let $Y$ be a curve with an action of the group
$$
V_4 = \langle r,s \;|\; r^2 = s^2 = (rs)^2 = 1 \rangle.
$$
Then we have the following diagram
\begin{equation} \label{diag1.4}
\xymatrix{
& Y \ar[d]^{a_r} \ar[dl]_{a_s} \ar[dr]^{a_{rs}} & \\
Y_s \ar[dr] & Y_r \ar[d] & Y_{rs} \ar[dl] \\
& Z &
}
\end{equation}
with $Y_v := Y/\langle v \rangle$ for any $v \in V_4$ and $Z= Y/V_4$. The following theorem is a special case of \cite[Theorem 3.2]{rr}.

\begin{prop} \label{prop1.6}
Suppose $a_r$ is \'etale, that $a_s$ respectively $a_{rs}$ are ramified at $2\alpha_s >0$,
respectively $2\alpha_{rs}>0$ points and $Z$ is of genus $g(Z)$. Then $P(Y_s/Z)$ and $P(Y_{rs}/Z)$
are subvarieties of $P(Y/Y_r)$ and the addition map
$$
\phi_r: P(Y_s/Z) \times P(Y_{rs}/Z)  \ra P(Y/Y_r)
$$
is an isogeny of degree $2^{2g(Z)}$.
\end{prop}

\section{A degree computation} \label{degcom}

As above, let $n=2^rm$ with $m$ odd and $r \geq 2$. Again we consider a
curve $X$ with action of the dihedral group 
$D_n := \langle  \sigma, \tau \;|\; \sigma^n = \tau^2 = (\sigma\tau)^2 = 1 \rangle.$
With the notation as in Section \ref{prel} we have  diagram \eqref{diag1.2} and
$s_0, s_1 \geq 2$. Then, apart from $f_1,f_2$ and $f_3$, all the maps in diagram \eqref{diag1.2}
are ramified. So the pullbacks of the corresponding Jacobians are embeddings. Recall that $P(f)$ denotes 
the Prym variety of the covering $f$.

We consider the isogenies  
$$
h:= \Nm a_{\tau\sigma^m} \circ  a^*_{\tau \sigma^{n/2}} : P(b_{\tau \sigma^{n/2}}) \lra P(b_{\tau\sigma^m}).
$$
and
$$
h' := \Nm a_{\tau \sigma^{n/2}} \circ  a^*_{\tau\sigma^m} : P(b_{\tau\sigma^m}) \lra P(b_{\tau \sigma^{n/2}}).
$$
Let
$$
A:=a^*_{\tau \sigma^{n/2}} (P(b_{\tau \sigma^{n/2}}) ) \qquad \mbox{and} \qquad
B:= a^*_{\tau\sigma^m} (P(b_{\tau\sigma^m}) )
$$ 
be subvarieties of $JX$.  
Now $\tau$ (respectively $\sigma^{n/2}$)  induces an involution on $X_{\tau \sigma^{n/2}}$ (respectively $X_{\tau \sigma^m})$, which we denote by the same letter. Thus the Prym variety $P(b_{\tau \sigma^{n/2}})$ is $\Ker (1+\tau)^0$
and $P(b_{\tau \sigma^m}) = \Ker (1+\sigma^{n/2})^0$. Hence we have 
(for example by \cite[Corollary 2.7]{rr}),
$$
A = \{ z \in JX^{\langle \tau \sigma^{n/2}\rangle}  \ \mid \ z + \tau z =0 \}^0,  \qquad B = \{ w \in JX^{\langle \tau\sigma^m \rangle}  \ \mid \ 
w + \sigma^{n/2} w =0 \}^0.
$$
Moreover, as in \cite{rr}, there is a commutative diagram:
\begin{equation} \label{prymdiag}
\xymatrix@R=1.3cm@C=1cm{
& A \ar[dr]_{\Nm {a_{\tau\sigma^m}}}  \ar[rr]^{1+\tau\sigma^m} &&  B \ar[dr]_{\Nm
 {a_{\tau\sigma^{n/2}}}}  \ar[rr]^{1+ \tau \sigma^{n/2}} && A \\
P(b_{\tau\sigma^{n/2}})  \ar[ur]^{a^*_{\tau\sigma^{n/2}}} \ar[rr]_h && 
 P(b_{\tau\sigma^m}) 
\ar[ur]^{a^*_{\tau\sigma^m}} \ar[rr]_{h'}
&& P(b_{\tau\sigma^{n/2}}) \ar[ur]^{a^*_{\tau\sigma^{n/2}}}  & 
}
\end{equation}

\begin{lem} \label{ker}
For any $n = 2^rm$ with $m$ odd and $r \geq 2$ we have 
$$
|\Ker h | = |\Ker(1+ \tau \sigma^m)_{|_A}|
$$
and 
$$
\Ker(1 + \tau \sigma^m)_A = (JX[2])^{\langle \tau, \sigma^m \rangle}.
$$
\end{lem}

\begin{proof}
The first assertion follows from diagram \eqref{prymdiag}, since 
$a^*_{\tau \sigma^{n/2}}: P(b_{\tau \sigma^{n/2}}) \ra A$ and
$a^*_{\tau \sigma^m}: P(b_{\tau \sigma^{m}}) \ra B$ are isomorphisms.
 For the last assertion note that
$z \in \Ker (1+\tau \sigma^m)_{|_A}$ if and only if
$$
\tau\sigma^{n/2} z = z,  \qquad \tau z=  -z \qquad \tau\sigma^m(z)=-z, 
$$
which implies that $\sigma^mz=z$. So
 $$
 z = \tau \sigma^{n/2} z = \tau \sigma^{2^{r-1}m} z = \tau (\sigma^m)^{2^{r-1}}(z) = \tau z = -z,
 $$
then $z \in A[2]$. 

Therefore $z \in \Ker (1+\tau \sigma^m)_{|_A}$ if and only if
$ z \in JX[2]$ such that $\tau z = z$ and $\sigma^m z =z$ which was to be shown.
\end{proof}

The following proposition is a generalization of a special case of \cite[Theorem 4.1,(ii)]{rr}.

\begin{prop} \label{p2.4}
For every $n=2^rm$ with $m$ odd and $r \geq 2$, we have 
$$
 \deg h = 2^{(m-1)(g-1)+s_1-2 }.
$$
\end{prop}

\begin{proof}
The proof is by induction on the exponent $r \geq 2$. Suppose first $r = 2$, i.e. $n =4m$.  Consider the curve $X$ with  the action of the dihedral subgroup 
$$
D_4 := \langle \sigma^{m}, \tau  \rangle \subset D_n.
$$
It has 2 non-conjugate Kleinian subgroups, namely 
 $K_\tau = \langle \sigma^{2m}, \tau \rangle$ and
 $K_{\tau \sigma^m} = \langle \sigma^{2m}, \tau \sigma^{m} \rangle$ 
Note that by \eqref{e1.2}, $T_\tau = T_{\tau \sigma^m}= \langle \sigma^{m}, \tau  \rangle$. Then according to
 \cite[Theorem 4.1,(ii)]{rr} 
$$
 |\Ker h| = 2^{2 g(X_{T_\tau}) -2 + s_1}.
$$
So Corollary \ref{cor1.4} give the proposition in this case.

Suppose now $r \geq 3$ and the proposition holds for $r-1$. Let $X$ be a curve with an action of $D_n$ with $X/\langle \sigma \rangle = H$, so that we 
have the diagram \eqref{diag1.2}. Then the subgroup $D_\frac{n}{2} = \langle \sigma^2, \tau \rangle$ of index 2 acts of the curve $X_{n/2}$, 
so that we can apply the inductive hypothesis.
This gives that the map 
$$
h_{n/2} := \Nm c_{\tau \sigma^m} \circ c_{\tau \sigma^{n/2}}^*: P(d_{\tau \sigma^{n/2}}) \lra P(d_{\tau \sigma^m})
$$
is an isogeny of degree $2^{(m-1)(g-1) -2 + s_1}$.

Hence it suffices to show that 
$$
\Ker h = b^*_{\tau \sigma^{n/2}}( \Ker h_{n/2}).
$$ 
This implies the proposition, since the map $b^*_{\tau\sigma^{n/2}}$ is injective.

Now Lemma \ref{ker} applied to the induction hypothesis, i.e. to $h_\frac{n}{2}$ gives
$$
|\Ker h_\frac{n}{2}| =|(JX_{\tau \sigma^{n/2}}[2])^{\langle \tau, \sigma^m \rangle}| =
	|a^*_{\tau \sigma^{n/2}}(JX_{\tau \sigma^{n/2}}[2])^{\langle \tau, \sigma^m \rangle}|.
$$
But 
\begin{eqnarray*}
a^*_{\tau \sigma^{n/2}}(JX_{\tau \sigma^{n/2}}[2])^{\langle \tau, \sigma^m \rangle} &=&
\{ z \in JX[2] \;|\; \tau \sigma^{n/2} z = z, \tau z = z, \sigma^mz=z \}\\
&=& \{ z \in JX[2] \;|\; \tau z = z, \sigma^mz=z \},
\end{eqnarray*}
since the equation $\tau \sigma^{n/2} z = z$ is a consequence of the last 2 equations.
This gives 
$$
|\Ker h| = |\Ker h_\frac{n}{2}|
$$ 
which completes the proof of Proposition \ref{p2.4}.
\end{proof}

\section{Decomposition for $n= 2^rm, \; r\geq 2$ with $m$ odd}

Now let the notation be as in Section 1 with $n = 2^r m, \; r \geq 2$ and $m$ odd. 
Let $f: X \ra H$ be a cyclic \'etale covering of degree $n$ of a hyperellitic curve $H$.
The main result of the paper is the following theorem.

\begin{thm} \label{thm4.1}
Let $n$ and $f: X \ra H$ be as above. Then
$JX_\tau$ and $JX_{ \tau \sigma^m}$ are abelian subvarieties of the Prym variety $P(f)$  and the addition map
$$
a: JX_\tau \times JX_{\tau \sigma^m} \ra P(f)
$$
is an isogeny of degree $  2^{[(2^{r} -r-1)m -(r-1)](g-1)}.$ 
\end{thm}

The proof is by induction on $r$. Since the proofs for $r=2$ and for the inductive step 
in case $r \geq 3$ are almost the 
same, we give them simultaneously. The difference is only that for $r = 2$ we use 
Theorem \ref{thm2.1} instead of the induction hypothesis.

So in the whole of this section we assume that for $r \geq 3$, Theorem \ref{thm4.1} is true for $r-1$, 
i.e. for covering of degree $2^{r-1}m$ for all $m$. Let $r \geq 2$ and $f: X \ra H$ be an \'etale covering of degree $n = 2^rm$ with odd $m \geq 1$.
We use the notation of diagram \eqref{diag1.2}. In addition let $b_\tau: X_\tau \ra X_{K_\tau}$ denote the canonical projection.

\begin{prop}  \label{prop4.2}
The varieties $JX_{K_{\tau}}, JX_{K_{\tau \sigma^m}}, P(b_{\tau})$ and 
$P(b_{\tau \sigma^{n/2}})$ are abelian 
subvarieties of $JX$ and the addition map 
$$
 \widetilde \phi_n: f^*JH \times JX_{K_{\tau}} \times JX_{K_{\tau \sigma^m}} \times 
P(b_{\tau}) \times P(b_{\tau \sigma^{n/2}}) \ra JX
$$
is an isogeny of degree
$$
\deg \widetilde \phi_n =  m^{2g-2} \cdot 2^{[(2^{r+1} -r)m +r](g-1) +2 -s_0}.
$$
\end{prop}

\begin{proof}
All maps in diagram \eqref{diag1.2} are ramified apart from $f_1, f_2$ and $f_3$, which gives the 
first assertion. 
The dihedral group $D_{n/2} = \langle \sigma^2, \tau \rangle$ acts on the curve $X_{\sigma^{n/2}}$. 

If $r =2$, we can apply Theorem \ref{thm2.1} to get that the canonical map
$$
\alpha: JX_{K_\tau} \times JX_{K_{\tau \sigma^m}} \ra P(f_3 \circ f_2)
$$
is an isomorphism. For $r \geq 3$ we can apply the induction hypothesis, which gives that $\alpha$ 
 is an isogeny of degree $2^{[(2^{r-1} - r)m - (r-2)](g-1)}$. Since this number is equal to 1 for $r = 2$, this is valid for all $r \geq 2$.

 Now the addition map 
$\alpha_1: (f_3 \circ f_2)^*JH \times JX_{K_\tau} \times JX_{K_{\tau \sigma^m}} \ra JX_{\sigma^{n/2}}$ factorizes as
$$
\xymatrix{
(f_3 \circ f_2)^*JH \times JX_{K_\tau} \times JX_{K_{\tau \sigma^m}} \ar[rr]^(0.6){\alpha_1} \ar[dr]_{id \times \alpha} && JX_{\sigma^{n/2}}\\
 &(f_3 \circ f_2)^*JH \times P(f_3 \circ f_2) \ar[ur]_(0.6){\psi} &  \\
}
$$
where $\psi$ is the addition map. So Lemma \ref{lem1.5} implies that
$$
\deg \alpha_1 = \deg \alpha \cdot \deg \psi = 2^{[(2^{r-1} - r)m - (r-2)](g-1)} \cdot (2^{r-1}m)^{2g-2} = m^{2g-2} \cdot 2^{[(2^{r-1} - r)m + r)](g-1)}.
$$
Clearly $\alpha_1$ and its  pullback via $f_1^*$  are  of the same degree. 
Moreover, considering $X$ with the action of the Klein group 
$\langle \sigma^{n/2}, \tau \rangle$, we have the diagram
$$
\xymatrix{
& X \ar[d]^{f_1} \ar[dl]_(0.4){a_\tau} \ar[dr]^{a_{\tau \sigma^{n/2}}}& \\
X_{\tau} \ar[dr]_{b_\tau} & X_{\sigma^{n/2}} \ar[d]^(0.4){c_{\tau \sigma^{n/2}}} & X_{\tau \sigma^{n/2}} 
\ar[dl]^{b_{\tau \sigma^{n/2}}} \\
& X_{K_\tau} &
}
$$
Then Proposition \ref{prop1.6} gives that the addition map
$$
\alpha_2:  P(b_\tau) \times P(b_{\tau \sigma^{n/2}}) \ra P(f_1)
$$
is an isogeny of degree $2^{2g(X_{K_\tau})} = 2^{2^{r-1}m(g-1) +2 -s_0}$.

Now note that the map $\widetilde \phi_n$ factorizes as
$$
\xymatrix@C=1.8cm@R=1.5cm{
  \left[f^*JH \times JX_{K_\tau} \times JX_{K_{\tau \sigma^m}} \right] \times \left[ P(b_\tau) 
\times P(b_{\tau \sigma^{n/2}}) \right] \ar[r]^-{\widetilde \phi_n} \ar[d]_{f_1^*\alpha_1 \times \alpha_2} & JX\\
 f_1^*JX_{\sigma^{n/2}} \times P(f_1) \ar[ur]_{\alpha_3} &  \\
}
$$
By Lemma \ref{lem1.5} the addition map $\alpha_3$ is an isogeny of degree $2^{2g(X_{\sigma^{n/2}})-2} = 
2^{2^rm(g-1)}$, therefore the map  $\widetilde \phi_n$ is an isogeny of degree 
\begin{eqnarray*}
\deg \widetilde \phi_n &=& \deg f_1^*\alpha_1 \cdot \deg \alpha_2 \cdot \deg \alpha_3 \\
&=& m^{2g-2} \cdot 2^{[(2^{r-1} - r)m + r)](g-1)} \cdot 2^{2^{r-1}m(g-1) +2 -s_0} \cdot 2^{2^rm(g-1)}\\ 
&=& m^{2g-2} \cdot 2^{[(2^{r+1} -r)m +r](g-1) +2 -s_0}.
\end{eqnarray*}
\end{proof}

\begin{cor} \label{cor4.3}
The canonical map
$$
 \phi_n: f^*JH \times JX_{K_\tau} \times JX_{K_{\tau\sigma^m}} \times P(b_\tau) 
\times P(b_{\tau \sigma^m}) \ra JX
$$
is an isogeny of degree 
$$
\deg \phi_n = m^{2g-2} 2^{[(2^{r+1} -r - 1)m +r-1](g-1)}.
$$
\end{cor}

\begin{proof}
According to Proposition \ref{p2.4} the canonical map 
$$
h : P(b_{ \tau \sigma^{n/2}}) \ra  P(b_{\tau \sigma^m})
$$
is an isogeny of degree $2^{(m-1)(g-1) -2 + s_1}$.  

Now with the definition of the map $h$ one checks that the following diagram commutes
$$
\xymatrix@C=1.5cm{
  \left[ f^*JH \times JX_{K_\tau} \times JX_{K_{\tau \sigma^m}}  \times  P(b_\tau) \right]
\times P(b_{\tau \sigma^{n/2}})  \ar[r]^-{\widetilde \phi_n} 
\ar[d]_{id \times h} & JX \\
   \left[f^*JH \times JX_{K_\tau} \times JX_{K_{\tau \sigma^m}}  \times  P(b_\tau) \right]
\times P(b_{\tau \sigma^m})  \ar[ur]_(0.6){\phi_n} & \\
}
$$
So Propositions \ref{prop4.2} and \ref{p2.4} imply that $\phi_n$ is an isogeny of degree 
\begin{eqnarray*}
\deg \phi_n &=& \frac{\deg \widetilde \phi_n}{\deg h}\\
&=&  \frac{m^{2g-2} \cdot 2^{[(2^{r+1} -r)m +r](g-1) +2 -s_0}}{2^{(m-1)(g-1) -2 + s_1}} 
= m^{2g-2} \cdot 2^{[(2^{r+1} -r - 1)m +r-1](g-1)}
\end{eqnarray*} 
where we used again that $s_0 + s_1 = 2g+2$.
\end{proof}

\begin{cor} \label{cor4.4}
The canonical map
$$
\psi_n: JX_{K_\tau} \times JX_{K_{\tau \sigma^m}} \times P(b_\tau)  \times
P(b_{\tau \sigma^m})
\ra P(f)
$$
is an isogeny of degree $2^{[(2^{r+1} -r - 1)m - (r+1)](g-1)}$.
\end{cor}

\begin{proof}
Clearly the addition maps the source of $\psi_n$ into $P(f)$ and the following diagram is 
commutative
$$
\xymatrix@C=1.5cm{
  f^*JH \times \left[ JX_{K_\tau} \times JX_{K_{\tau \sigma^m}} \times P(b_\tau)  \times
P(b_{\tau \sigma^m}) \right] \ar[r]^-{ \phi_n} 
\ar[d]_{id \times \psi_n} & JX\\
 f^*JH  \times  P(f) \ar[ur]_{\varphi} & \\
}
$$
where $\varphi$ denotes the addition map. According to Lemma \ref{lem1.5}, $\varphi$ is an isogeny of degree $(16m)^{2g-2}$. 
Hence $\psi_n$ is an isogeny of degree 
$$
\deg \psi_n = \frac{\deg \phi_n}{\deg \varphi} = 
\frac{m^{2g-2} \cdot 2^{[(2^{r+1} -r - 1)m +r-1](g-1)}}{(2^rm)^{2g-2}} =
 2^{[(2^{r+1} -r - 1)m - (r+1)](g-1)}
$$
\end{proof}

\begin{proof}[Proof of Theorem 4.1]
The following diagram is commutative
$$
\xymatrix{
 JX_{K_\tau} \times JX_{K_{\tau \sigma^m}} \times P(b_\tau)  \times
P(b_{\tau \sigma^m})  \ar[dr]^(0.6){ \psi_n} 
\ar[d]_{\simeq} & \\
\left[ JX_{K_\tau} \times P(b_\tau)  \right] \times \left[JX_{K_{\tau \sigma^m}}
\times 
P(b_{\tau \sigma^m}) \right] 
\ar[d]_{\varphi_1 \times \varphi_2}  & P(f)\\
JX_\tau \times JX_{\tau \sigma^m} \ar[ur]_{a} & \\
}
$$
where $\varphi_1$ and $\varphi_2$ denote the addition maps. According to Lemma \ref{lem1.5}, 
$\varphi_1$ and $\varphi_2$ are isogenies of degrees
$2^{8m(g-1)  +2 - s_0}$ and $2^{8m(g-1) +2 - s_1}$ respectively. 
This implies that $a$ is an isogeny of degree
$$
\deg a = \frac{\deg \psi_n}{\deg \varphi_1 \cdot \deg \varphi_2} 
= \frac{2^{[(2^{r+1} -r - 1)m - (r+1)](g-1)}}{2^{(2^rm -2)(g-1)}} = 2^{[(2^r -r -1)m - (r-1)](g-1)}.
$$
which completes the proof of the theorem.
\end{proof}

\end{document}